\documentclass{article}[11pt]
\usepackage[a4paper]{geometry}
\usepackage[utf8]{inputenc}
\usepackage[T1]{fontenc}
\usepackage[british]{babel}
\usepackage[autostyle]{csquotes}
\usepackage{mathtools}
\usepackage{amsmath, amsthm,amssymb,amstext,amsbsy} 
\usepackage{tabularx}
\usepackage{thmtools,thm-restate}
\usepackage{fullpage}
\usepackage{todonotes}[color=white]
\usepackage{enumitem}
\usepackage{caption}
\setlist[description]{leftmargin=15pt,labelindent=15pt}

\usepackage[bookmarks=false]{hyperref}  
\definecolor{darkgreen}{rgb}{0,0.5,0}
\hypersetup{linktocpage=true,colorlinks=true,citecolor=red,linkcolor=darkgreen} 
\newtheorem{theorem}{Theorem}

\newtheorem{propos}[theorem]{Proposition}
\newtheorem{lemma}[theorem]{Lemma}

\newcommand{\Z}{\mathbb{Z}}
\newcommand{\R}{\mathbb{R}}

\newcommand{\rev}[1]{{#1}^{\REV}}

\DeclareMathOperator{\REV}{R}

\newcommand{\cD}{\mathcal{D}}
\newcommand{\cO}{\mathcal{O}}
\newcommand{\cS}{\mathcal{S}}

\title{Min-cost-flow preserving bijection\\ between subgraphs and orientations}
\author{Izhak Elmaleh\thanks{Hebrew University of Jerusalem Israel, \texttt{izhak.elmaleh@gmail.com}.} \and Ohad N. Feldheim\thanks{Hebrew University of Jerusalem Israel, \texttt{ohad.feldheim@mail.huji.ac.il}. Supported by ISF grant 1327/19.} }

\date{}

\begin{document}

\maketitle

\begin{abstract}
\vspace{0.1em}
Consider an undirected graph $G=(V,E)$.  A subgraph of $G$ is a subset of its edges, whilst an orientation of $G$ is an assignment of a direction to each edge. 
Provided with an integer circulation-demand $d:V\to \Z$, we show an explicit and efficiently computable bijection between subgraphs of $G$ on which a $d$-flow exists and orientations on which a $d$-flow exists. Moreover, given a cost function $w:E\to (0,\infty)$ we can find such a bijection which preserves the $w$-min-cost-flow.

In 2013, Kozma and Moran \cite{KM13} showed, using dimensional methods, that the number of subgraphs $k$-connecting a vertex $s$ to a vertex $t$ is the same as the number of orientations $k$-connecting $s$ to $t$. An application of our result is an efficient, bijective proof of this fact.
\end{abstract}


\section{Introduction}
Let $G=(V,E)$ be a simple graph with positive cost function $w:E\to (0,\infty)$. We regard 
$G$ as a digraph, treating every undirected edge as a pair of directed edges in reverse direction. 

Denote the set of \emph{subgraphs} of $G$ by $\cS(G)=\left\{K\subset E\ :\ \forall_{e\in E}|\{e,\rev{e}\}\cap K|\neq 1\right\}$, and the set of \emph{orientations} of $G$ by $\cO(G)=\left\{L\subset E\ :\ \forall_{e\in E}|\{e,\rev{e}\}\cap L|=1\right\}$, where $\rev{(u,v)}=(v,u)$.

A function $d:V\to \mathbb{Z}$ such that $\sum_{u\in V}d(u)=0$ is called an \emph{integer demand}. A $d$-flow on $G$ is function $f:E\to[0,1]$ satisfying such that for any $u\in V$ we have $\sum_{v\sim u}f((u,v))-f((v,u))=d(u)$. We say that $f$ is a flow on a directed subgraph $D\subset E$ if $f(e)=0$ for all $e\notin D$. Denote $\cS_f$, $\cO_f$, $\cD_f$ for the set of subgraphs, orientations and directed subgraphs of $G$ on which a $d$-flow exists. Given a cost function $w:E\to \R_+$ and a flow $f$, write $|f|_w=\sum_{e\in E}w(e)|f(e)|$ for the total $w$-cost of $f$. The $w$ min-cost-flow satisfying $d$ is the $d$-flow for which this cost is minimal. 

Our main result is the following.
\begin{theorem}\label{thm:main-tech}
For any graph $G$, integer demand $d$ and cost function $w$,
there exists an explicit bijection between $\cS_f$ and $\cO_f$, computable in polynomial time, which preserves a $w$ min-cost-flow.
\end{theorem}

We call a path from $s$ to $t$ in $G$ an $(s,t)$-path. A directed graph is said to \emph{$k$-connect $s$ and $t$} if there exist $k$ disjoint $(s,t)$-directed paths. Denote $\cS_k$ and $\cO_k$ the sets of subgraphs and orientations of $G$ which $k$-connect $s$ and $t$, respectively.

A collection of $k$ disjoint $(s,t)$-directed paths is called \emph{minimal} in a directed graph $D$ if the total weight of edges participating in the paths is minimal. Recalling the classical Integrality Theorem, which guarantees that any min-cost-flow problem in a graph (i.e. with capacity $1$ for each edge) has an integer optimal solution, 
Theorem~\ref{thm:main-tech}
implies the following.

\begin{theorem}\label{thm:main}
For any weighted graph $G=(V,E,w)$ There exists an explicit bijection between $\cS_k$ and $\cO_k$, computable in polynomial time, which preserves the minimal $k$ collection of $w$-shortest paths.
\end{theorem}

The result could be easily generalised to vertex disjoint paths by introducing vertex capacities. 

\section{Background and motivation}

The study of the relationship between subgraphs and orientation is a classical subject in combintorics. In 1960 Nash-Willims \cite{nash1960orientations}, generalizing a 1939 result by Robbins \cite{robbins1939theorem}, showed that every undirected graph $G$ has a well-balanced orientation. Chvátal and Thomassen \cite{chvatal1978distances} proved that every undirected bridgeless (i.e. 2-connected) graph of radius $r$ admits an orientation of radius at most $r^{2r}+r$, and that this bound is best possible. In \cite{Bernardi08} Bernardi showed that evaluation of the Tutte polynomial counts both the spanning subgraphs and the orientations of $G$. 

In 2013, Kozma and Moran \cite{KM13}, introduced Vapnik–Chervonenkis (VC) theory to the subject. They showed that there are several properties $\phi$ of graphs, for which number of subgraphs of a given graph $G$ which satisfy $\phi$ is either the same, or dominates the number of orientations satisfying it. Their proof relies upon shattering extremal systems, using the sandwich theorem \cite{paj85}. Recently Bucić, Janzer and Sudakov \cite{BJS20} used this method to count $H$-free orientations of a given graph G.

In particular, it was shown in \cite{KM13} that $|\cS_k|=|\cO_k|$. Their method, however, is non-constructive, and its naïve algorithmic application is of exponential complexity in $|E|$. In Theorem~\ref{thm:main} we obtain an explicit, natural and efficiently computable bijection between $\cS_k$ and $\cO_k$, which preserves a particular collection of $k$-disjoint paths. 


\subsection{Notation and conventions}

Throughout $G=(V,E)$, the base graph, $w:E\to \R_+$, the weight function, $d:V\to \Z$, the demand function, $E'\in \cO(E)$ an arbitrary orientation of $G$ and an \emph{a priori}
order $e_1,\dots, e_{|E|}$ on the edges of $E$ are all fixed. 
We also define \[\chi(e)=\begin{cases}
e & \text{for }e\in E',\\
\rev{e} & \text{for } \rev{e} \in E'.
\end{cases}\]

For simplicity, but without loss of generality, we assume that $G$ has no two distinct paths of equal length.

We denote the solution to the min-cost-flow problem in the directed subgraph $D\subset E$ with respect to $d$ by $A(D)$, whenever such a solution exists. Using the Integrality Theorem, we treat $A(D)$ both as a set of directed edges and as a flow.

To simplify addition and subtraction of edges from a directed subgraph we employ the orientation operation $D\oplus e:=\{D\cup\{e\}\}\setminus\rev{e}$, the symmetric inclusion operation, $E+e:=E\cup\{e, \rev{e}\}$ and the symmetric exclusion operation $E-e:=E\setminus\{e, \rev{e}\}$.

\section{The bijection}

Our bijection relies on the following lemma.

\begin{lemma}\label{lem:properbij}
Let $D\in \cD_f$ and $e\in G$. Then at least one of the following holds:
\begin{itemize}
    \item $A(D\oplus e)=A(D)$,
    \item $A(D\oplus \rev{e})=A(D)$.
\end{itemize}
\end{lemma}
\begin{proof}
If either $e\in D$, $\rev{e}\in D$ or both, the lemma is straightforward, as $A(D)$ cannot include both $e$ and $\rev{e}$.
We may therefore assume that $\{e,\rev{e}\}\cap D=\emptyset$.

The classical Integrality Theorem, guarantees that any min-cost-flow problem in a graph (i.e. with capacity $1$ for each edge) has an integer optimal solution. Write $F_0$ for the min-cost $d$-flow in $D$, $F_1$ for 
 the min-cost $d$-flow in $D\cup \{e\}$, and  $F_2$  for  the min-cost $d$-flow  in $D\cup \{\rev{e}\}$. 

Assume for the sake of obtaining a contradiction that these three flows are distinct, so that, by monotonicity, $|F_1|_w,|F_2|_w<|F_0|_w$. In particular, this implies, by the Integrality Theorem, that $F_1$ assigns flow $1$ to $e$ and $F$ assigns flow $1$ to $\rev{e}$ or vice versa, as otherwise one of these flows would exist also as a flow on $D$. This implies, however, that $\frac{F_1+F_2}{2}$, which is also a $d$-flow, assigns a total of $0$ flow to $e$, so that it is a proper flow on $D$. Clearly, the total weight of this flow is less than the maximum among $F_1$ and $F_2$, a contradiction to the minimality of $F_0$. We deduce that either $F_0=F_1$ or $F_0=F_2$.
\end{proof}
\pagebreak
We also require the following observation.
\begin{lemma}\label{lem:one direct}
Let $D\in \cD_f$ and $e\in A(D)$, then 
$A(D)=A(D+e)$.
\end{lemma}
\begin{proof}
Assume to the contrary that $A(D+e)\neq A(D)=A(D\oplus e)$ and denote by $F_0$ the flow corresponding to $A(D)$ and $F_1$ for the flow corresponding to $A(D+e)$, as in the proof of lemma~\ref{lem:properbij}. By monotonicity, we have $|F_1|_w<|F_0|_w$. By minimality this implies that $F_1(\rev{e})=1$ whilst by our assumption 
$F_0(e)=1$. Hence $\frac{F_0+F_1}2$ is a flow on $D$ satisfying $|\frac{F_0+F_1}2|<|F_1|$, a contradiction. Therefore $A(D)=A(D+e)$.
\end{proof}

Equipped with Lemma~\ref{lem:properbij}, the orientation $E'$ and our order $(e_1,\dots, e_{|E|})$, we are ready to present our bijection in the next couple of sections.

\subsection{The $\cS_f\to\cO_f$ bijection}

The bijection $\phi:\cS_f\to\cO_f$ is iteratively obtained  by firstly orienting $e_1$, then $e_2$ and so forth.
This is done by applying a sequence of maps $\phi_i:\cD_f\to\cD_f$
for $i\in\{1,\dots, |E|\}$, such that for all $i\le |E|$ we have $\phi_i(D)\setminus\{e_i,\rev{e}_i\}=D\setminus\{e_i,\rev{e}_i\}$ and $A(D)=A(\phi_i(D))$. Define $\phi_i(D)$ as follows.

Firstly, we make sure that the orientation of $e_i$ will not alter the min-cost $d$-flow. 
\begin{enumerate}
    \item if $A(D)\neq A\big(D\oplus e_i\big)$ we set $\phi_i(D):=D\oplus \rev{e}_i$,
    \item if $A(D)\neq A\big(D\oplus \rev{e}_i\big)$ we set $\phi_i(D):=D\oplus
     e_i$.
\end{enumerate}
In the remaining case, where $A(D)= A(D\oplus \{e_i\})=A\big(D\oplus \{\rev{e}_i\}\big)$ we do the following:
\begin{enumerate}
    \item[3.] if $e_i\in D$ we set $\phi_i(D):=D\oplus \chi(e_i)$,
    \item[4.] if $e_i\notin D$ we set $\phi_i(D):=D\oplus \rev{\chi(e_i)}$.
\end{enumerate}
Observe that, by Lemma~\ref{lem:properbij}, rules (1.) and (2.) are mutually exclusive so that $\phi_i$ is well defined, $A(D)=A(\phi_i(D))$ and $|\phi_i(D)\cap \{e_i,\rev{e}_i\}|=1$.

We then set $\phi(K):=\phi_{|E|} \circ \phi_{|E-1|} \circ \cdots \circ \phi_1(K)$ so that $\phi$ maps $S_f$ to $O_f$ and $A(K)=A(\phi(K))$.

\subsection{The $\cO_f\to\cS_f$ bijection}
The reverse bijection $\psi:\cO_f\to\cS_f$ is obtained similarly. This time we iterate by first deciding whether to include both $e_{|E|}$ \& $\rev{e}_{|E|}$ or neither of them, then $e_{|E|-1}$ \& $\rev{e}_{|E|-1}$ and so forth. This is done by applying a sequence of maps $\psi_i:\cD_f\to \cD_f$ for $i\in \{1,\dots, |E|\}$, such that for all $i\le |E|$ we have 
$\psi_i(D)\setminus\{e_i,\rev{e}_i\}=D\setminus\{e_i,\rev{e}_i\}$.

Firstly, we verify that the decision to include or exclude $e_i$ \& $\rev{e}_i$ will not alter the min-cost $d$-flow.
\begin{enumerate}
    \item if $A(D)\neq A\big(D+ e_i\big)$ we set $\psi_{i}(D):=D-e_i$,
    \item if $A(D)\neq A\big(D- e_i\big)$ we set $\psi_{i}(D):=D+e_i$.
\end{enumerate}
In the remaining case, where $A(D)= A(D+ e_i)=A(D-e_i)$, we do the following:
\begin{enumerate}
    \item[3.] if $\chi(e_i)\in D$ we set $\psi_{i}(D):=D+e_i$,
    \item[4.] if $\chi(e_i)\notin D$ we set $\psi_{i}(D):=D-e_i$.
\end{enumerate}
Observe that, by Lemma~\ref{lem:one direct} rules (1.) and (2.) are mutually exclusive, as the former's condition is impossible if $\{e_i,\rev{e}_i\}\cap A(D)\neq \emptyset$ and the latter's is impossible otherwise. Hence that $\psi_i$ is well defined, $A(D)=A(\psi_i(D))$ and $|\psi_i(D)\cap \{e_i,\rev{e}_i\}|\neq 1$.

We then set $\psi(L):=\psi_{1} \circ \psi_{2} \circ \cdots \circ \psi_{|E|}(L)$ so that $\psi$ maps $O_f$ to $S_f$ and $A(L)=A(\psi(L))$.

\section{Proof of bijectivity}
In this section we establish the fact that $\psi$ is the inverse function of $\phi$ and, as a consequence, Theorem~\ref{thm:main-tech}.

This is an immediate consequence of the following

\begin{propos}\label{prop:tech}
For all $D\in \cD_f$ it holds that 
\begin{itemize}
    \item if $|\psi_i(D)\cap \{e_i,\rev{e}_i\}|\neq 1$ then $D=\psi_i\circ \phi_i(D)$
    \item if $|\psi_i(D)\cap \{e_i,\rev{e}_i\}|= 1$ then $D=\phi_i\circ \psi_i(D)$
\end{itemize}
\end{propos}
\begin{proof}
Firstly, we show three statements, 
\begin{equation}\label{eq:equiv1}
\begin{gathered}
\text{either $A(D)= A(D\oplus  e_i)$ or $A(D)=A(D\oplus \rev{e}_i)$},\\
\text{either $A(D)= A(D+ e_i)$ or $A(D)=A(D-e_i)$},\\
\text{$A(D)= A(D+ e_i)=A(D-e_i)$ if and only if 
$A(D)= A(D\oplus e_i)=A(D\oplus \rev{e}_i)$.}
\end{gathered}
\end{equation}
The first two observations are immediate from Lemma~\ref{lem:properbij} and Lemma~\ref{lem:one direct}, respectively. 
To see the last equivalence, observe that the same two lemmata imply that the statements $A(D+ e_i)\neq A(D-e_i)$ and $A(D\oplus e_i)\neq A(D \oplus \rev{e}_i)$ are both equivalent to the fact that $e_i\in A(D+ e_i)$ or $\rev{e}_i\in A(D+e_i)$.

Using \eqref{eq:equiv1} we deduce that $\psi_i(D)$,$\phi_i(D)$,$\psi_i\circ \phi_i(D)$ and $\phi_i\circ \psi_i(D)$, are either all produced by rules (1.) and (2.) of their respective definition, or all produce by rules (3.) and (4.).

If they are all produced by rules (1.) and (2.), then we consider three cases. 
\begin{itemize}
    \item $\mathbf{\{e_i,\rev{e}_i\}\subset D}$\textbf{.} In this case 
    $A(D- e_i)\neq A(D)$ so that either $e_i\in A(D)$ or $\rev{e}_i\in A(D)$. Assume without loss of generality $e_i\in A(D)=A(\phi(D))$, so that, by rule (1.) of the definition of $\phi_i$ we have $\phi_i(D)=\phi_i(D)\oplus e_i$ and we obtain $\psi_i \circ \phi_i(D)=(\phi_i(D)\oplus e_i)+e_i=D$, by rule (2.) of the definition of $\psi_i$.
    \item $\mathbf{\{e_i,\rev{e}_i\}\cap D=\emptyset}$\textbf{.} In this case
    $A(D+ e_i)\neq A(D)$ so that either $e_i\in A(D+e_i)$ or $\rev{e}_i\in A(D+e_i)$. Assume without loss of generality $e_i\in A(D+e_i)\neq A(\phi(D))$, so that, by rule $(1.)$ of the definition of $\phi_i$ we have $\phi_i(D)=\phi_i(D)\oplus \rev{e}_i$ and we obtain $\psi_i \circ \phi_i(D)=(\phi_i(D)\oplus e_i)+e_i=D$, by rule (1.) of the definition of $\psi_i$.
    \item $\mathbf{|\{e_i,\rev{e}_i\}\cap D|=1}$\textbf{.} In this case, assume without loss of generality $e_i\in D$. \\If $A(D + e_i)\neq A(D)$ then $\rev{e}_i\in A(D+e_i)$ so that $A(D+e_i)\neq A(D)$. Thus by rule (1.) of the definition of $\psi$, we have $\psi(D)=D-e_i$ and, by rule (2.) of the definition of $\phi_i$, we obtain $\phi_i \circ \psi_i(D)=(\phi_i(D)- e_i)\oplus e_i=D$.
    \\If, on the other hand, $A(D- e_i)\neq A(D)$ then $e_i\in A(D)$ so that $A(D- e_i)\neq A(D)$. Thus by rule (2.) of the definition of $\psi$, we have $\psi(D)=D+e_i$ and, by rule (2.) of the definition of $\phi_i$, we obtain $\phi_i \circ \psi_i(D)=(\phi_i(D)+ e_i)\oplus e_i=D$.
\end{itemize}
If 
$\psi_i(D)$,$\phi_i(D)$,$\psi_i\circ \phi_i(D)$ and $\phi_i\circ \psi_i(D)$  are all produced by rules (3.) and (4.), then 
\begin{itemize}
    \item If $\{e_i,\rev{e}_i\}\subset D$ then $\phi(D)=D\oplus \chi(e_i)$ and $\psi_i\circ \phi_i(D)= (D\oplus \chi(e_i))+e_i=D$.
    \item If $\{e_i,\rev{e}_i\}\cap D=\emptyset$ then
    $\phi(D)=D\oplus \rev{\chi(e_i)}$ and $\psi_i\circ \phi_i(D)= (D\oplus \rev{\chi(e_i)})-e_i=D$.
    \item If $\{e_i,\rev{e}_i\}\cap D=\{\chi(e_i)\}$ then
    $\psi(D)=D+e_i$ and $\phi_i\circ \psi_i(D)= (D+e_i)\oplus \chi(e_i)=D$.
    \item If $\{e_i,\rev{e}_i\}\cap D=\{\rev{\chi(e_i)}\}$ then
    $\psi(D)=D-e_i$ and $\phi_i\circ \psi_i(D)= (D-e_i)\oplus \rev{\chi(e_i)}=D$.
\end{itemize}
\end{proof}

\section{Complexity}
As for Theorem~\ref{thm:main}, finding the minimal $k$ disjoint $(s,t)$-directed path could be done efficiently using the Suurballe algorithm \cite{suurballe84}, an extension of the Dijxtra algorithm \cite{dijkstra59}. The worst case complexity of this algorithm is $O(k|E|+k|V|\log |V|)$. The general case of Theorem~\ref{thm:main-tech}, has the complexity of solving the min-cost-flow problem, i.e., $O(|E||V|^2\log |V|)$.
\section*{Acknowledgment}
We wish to thank Yuval Peled for his suggestions concerning the presentation of the paper and for useful discussions.

\end{document}